\documentclass[12pt,twoside]{amsart}
 \usepackage[latin1]{inputenc}
\usepackage{amsmath, amsthm, amscd, amsfonts, amssymb, graphicx}
\usepackage[bookmarksnumbered, plainpages]{hyperref}

\newtheorem{thm}{Theorem}[section]

\newtheorem{lem}[thm]{Lemma}
\newtheorem{prop}[thm]{Proposition}
\newtheorem{defn}[thm]{Definition}

\numberwithin{equation}{section}

\begin{document}

\title{\bf  The two-variable elliptic genus in odd dimensions }
\author{Yong Wang}

\thanks{{\scriptsize
\hskip -0.4 true cm \textit{2010 Mathematics Subject Classification:}
58C20; 57R20; 53C80.
 }}

\maketitle

\begin{abstract}
A kind of two-variable elliptic genus for almost-complex manifolds was introduced by Ping Li and its various properties were established by him. In this paper, we define a two-variable elliptic genus for odd dimensional spin manifolds which is the index for some Toeplitz operator and a holomorphic $SL(2,Z)$-Jacobi form. We also define some two-variable elliptic genera for almost-complex manifolds and odd dimensional spin manifolds which are holomorphic $\Gamma_0(2)$, $\Gamma^0(2)$, $\Gamma _{\theta}$-Jacobi forms. By these Jacobi forms, we can get some $SL(2,{\bf Z})$ and $\Gamma^0(2)$ modular forms. By these $SL(2,{\bf Z})$ and $\Gamma^0(2)$ modular forms, we get some interesting
anomaly cancellation formulas for almost complex manifolds and odd spin manifolds. As corollaries, we get some divisibility results of the holomorphic Euler characteristic number and the index of Toeplitz operators. In addition, we also define some another two-variable elliptic genera for even (rep. odd ) dimensional manifolds which are meromorphic $\Gamma_0(2)$, $\Gamma^0(2)$, $\Gamma _{\theta}$-Jacobi forms.
\end{abstract}

\vskip 0.2 true cm


\pagestyle{myheadings}
\markboth{\rightline {\scriptsize Yong Wang}}
         {\leftline{\scriptsize  The two-variable elliptic genus in odd dimension}}

\bigskip
\bigskip


\section{ Introduction}
For an arbitrary compact spin manifold one can define its elliptic genus. It is a modular form in one variable with respect to a congruence
subgroup of level $2$. For a compact complex manifold one can define its elliptic genus as a function in two complex variables. In the last case,
the elliptic genus is the holomorphic Euler characteristic of a formal power series with vector bundle coefficients. If the first Chern class of the complex manifold is equal to zero, then the elliptic genus is a weak Jacobi form. In \cite{LiP}, Li extended the elliptic genus of an almost complex
manifold to a twisted version where an extra complex vector bundle is involved. Under some conditions, Li proved this elliptic genus is a weak Jacobi
form. In \cite{Wa1}, We extended the Li's elliptic genus and proved this generalized elliptic genus is not a weak Jacobi form and called it the generalized Jacobi form. Moreover, by this generalized Jacobi form, we got some new $SL(2,\mathbb{Z})$ modular forms. In \cite{Gr}, Gritsenko defined the modified Witten genus or the automorphic correction of elliptic genus of an arbitrary holomorphic vector bundle over a compact complex manifold, and constructed a new object $\hat{A}^{(2)}_2$-genus. Furthermore, Gritsenko proved under certain conditions that this $\hat{A}^{(2)}_2$-genus is a weak Jacobian form. In \cite{GLW}, we extended Gritsenko's $\hat{A}^{(2)}_2$-genus and proved that this generalized $\hat{A}^{(2)}_2$-genus is not a weak Jacobian form, but a generalized Jacobian form. By this generalized Jacobi form, we can get some new $SL(2,\mathbb{Z})$ modular forms as in \cite{Wa1}.\\
\indent On the other hand, in \cite{Li1}, Liu introduced a modular form of a 4k-dimensional spin manifold with a weight of 2k. In \cite{CHZ}, Chen, Han and Zhang defined an integral modular form of weight $2k$ for a $4k$-dimensional spin manifold and an integral modular form of weight $2k$ for a $4k+2$-dimensional spin manifold. In \cite{LW1} and \cite{LW2}, Liu and Wang go through studying modular invariance properties of some characteristic forms to get some new anomaly cancellation formulas on $(4k-1)$ dimensional manifolds. And they derive some results on divisibilities on $(4k-1)$ dimensional spin manifolds and congruences on $(4k-1)$ dimensional ${\rm spin^c}$ manifolds. In \cite{GWL}, Guan, Wang and Liu introduced some modular forms of weight $2k$ over $SL_{2}(\bf Z)$ and some modular forms of weight $2k$ over $\Gamma^0(2)$ and $\Gamma_0(2)$ in odd dimensions respectively through the $SL(2,{\bf Z})$ modular forms introduced in \cite{Li2} and \cite{CHZ}. In parallel, we derived some new anomaly cancellation formulas and some divisibility results over spin manifolds and ${\rm spin^c}$ manifolds in odd dimensions.\\
  \indent So a natural question is to how to define a two-variable elliptic genus for odd dimensional manifolds and by this two-variable elliptic genus derive some new anomaly cancellation formulas. In this paper, we solve this problem and we define a two-variable elliptic genus for odd dimensional spin manifolds which is the index for a Toeplitz operator and a holomorphic $SL(2,Z)$-Jacobi forms. We also define some two-variable elliptic genera for almost-complex manifolds and odd dimensional spin manifolds which are holomorphic $\Gamma_0(2)$, $\Gamma^0(2)$, $\Gamma _{\theta}$-Jacobi forms. By these Jacobi forms, we can get some $SL(2,{\bf Z})$ and $\Gamma^0(2)$ modular forms. By these $SL(2,{\bf Z})$ and $\Gamma^0(2)$ modular forms, we get some interesting
anomaly cancellation formulas for almost complex manifolds and odd spin manifolds. As corollaries, we get some divisibility results of the holomorphic Euler characteristic number and the index of Toeplitz operators. In addition, we also define some another two-variable elliptic genera for even (rep. odd ) dimensional manifolds which are meromorphic $\Gamma_0(2)$, $\Gamma^0(2)$, $\Gamma _{\theta}$-Jacobi forms.\\
 \indent This paper is organized as follows. In Section 2, we have introduced some definitions and basic concepts that we will use in the paper. In Section 3, we define a two-variable elliptic genus for odd dimensional spin manifolds. By this Jacobi form, we can get some interesting anomaly cancellation formulas for odd dimensional spin manifolds. In Section 4, we define some two-variable elliptic genera for almost-complex manifolds and odd dimensional spin manifolds which are holomorphic $\Gamma_0(2)$, $\Gamma^0(2)$, $\Gamma _{\theta}$-Jacobi forms. By these Jacobi forms, we get some interesting
anomaly cancellation formulas for almost complex manifolds and odd ${\rm spin}$ manifolds. In Section 5, we define some another two-variable elliptic genera for even (rep. odd ) dimensional manifolds which are meromorphic $\Gamma_0(2)$, $\Gamma^0(2)$, $\Gamma _{\theta}$-Jacobi forms.\\

 \vskip 1 true cm

 \section{Characteristic Forms and Modular Forms}
\quad The purpose of this section is to review the necessary knowledge on
characteristic forms and modular forms that we are going to use.\\

 \noindent {\bf  2.1 characteristic forms }\\
 \indent Let $M$ be a Riemannian manifold.
 Let $\nabla^{ TM}$ be the associated Levi-Civita connection on $TM$
 and $R^{TM}=(\nabla^{TM})^2$ be the curvature of $\nabla^{ TM}$.
 Let $\widehat{A}(TM,\nabla^{ TM})$ defined by (cf. \cite{Zh})
\begin{equation}
   \widehat{A}(TM,\nabla^{ TM})={\rm
det}^{\frac{1}{2}}\left(\frac{\frac{\sqrt{-1}}{4\pi}R^{TM}}{{\rm
sinh}(\frac{\sqrt{-1}}{4\pi}R^{TM})}\right).
\end{equation}
 Let $E$, $F$ be two Hermitian vector bundles over $M$ carrying
   Hermitian connection $\nabla^E,\nabla^F$ respectively. Let
   $R^E=(\nabla^E)^2$ (resp. $R^F=(\nabla^F)^2$) be the curvature of
   $\nabla^E$ (resp. $\nabla^F$). If we set the formal difference
   $G=E-F$, then $G$ carries an induced Hermitian connection
   $\nabla^G$ in an obvious sense. We define the associated Chern
   character form as
   \begin{equation}
       {\rm ch}(G,\nabla^G)={\rm tr}\left[{\rm
   exp}(\frac{\sqrt{-1}}{2\pi}R^E)\right]-{\rm tr}\left[{\rm
   exp}(\frac{\sqrt{-1}}{2\pi}R^F)\right].
   \end{equation}
   For any complex number $t$, let
   $$\wedge_t(E)={\bf C}|_M+tE+t^2\wedge^2(E)+\cdots,~S_t(E)={\bf
   C}|_M+tE+t^2S^2(E)+\cdots$$
   denote respectively the total exterior and symmetric powers of
   $E$, which live in $K(M)[[t]].$ The following relations between
   these operations hold,
   \begin{equation}
       S_t(E)=\frac{1}{\wedge_{-t}(E)},~\wedge_t(E-F)=\frac{\wedge_t(E)}{\wedge_t(F)}.
   \end{equation}
   Moreover, if $\{\omega_i\},\{\omega_j'\}$ are formal Chern roots
   for Hermitian vector bundles $E,F$ respectively, then
   \begin{equation}
       {\rm ch}(\wedge_t(E))=\prod_i(1+e^{\omega_i}t)
   \end{equation}
   Then we have the following formulas for Chern character forms,
   \begin{equation}
       {\rm ch}(S_t(E))=\frac{1}{\prod_i(1-e^{\omega_i}t)},~
{\rm ch}(\wedge_t(E-F))=\frac{\prod_i(1+e^{\omega_i}t)}{\prod_j(1+e^{\omega_j'}t)}.
   \end{equation}
\indent If $W$ is a real Euclidean vector bundle over $M$ carrying a
Euclidean connection $\nabla^W$, then its complexification $W_{\bf
C}=W\otimes {\bf C}$ is a complex vector bundle over $M$ carrying a
canonical induced Hermitian metric from that of $W$, as well as a
Hermitian connection $\nabla^{W_{\bf C}}$ induced from $\nabla^W$.
If $E$ is a vector bundle (complex or real) over $M$, set
$\widetilde{E}=E-{\rm dim}E$ in $K(M)$ or $KO(M)$.\\

\noindent{\bf 2.2 Some properties about the Jacobi theta functions
and modular forms}\\
   \indent We first recall the four Jacobi theta functions are
   defined as follows( cf. \cite{Ch}):
   \begin{equation}
      \theta(v,\tau)=2q^{\frac{1}{8}}{\rm sin}(\pi
   v)\prod_{j=1}^{\infty}[(1-q^j)(1-e^{2\pi\sqrt{-1}v}q^j)(1-e^{-2\pi\sqrt{-1}v}q^j)],
   \end{equation}
\begin{equation}
    \theta_1(v,\tau)=2q^{\frac{1}{8}}{\rm cos}(\pi
   v)\prod_{j=1}^{\infty}[(1-q^j)(1+e^{2\pi\sqrt{-1}v}q^j)(1+e^{-2\pi\sqrt{-1}v}q^j)],
\end{equation}
\begin{equation}
    \theta_2(v,\tau)=\prod_{j=1}^{\infty}[(1-q^j)(1-e^{2\pi\sqrt{-1}v}q^{j-\frac{1}{2}})
(1-e^{-2\pi\sqrt{-1}v}q^{j-\frac{1}{2}})],
\end{equation}
\begin{equation}
   \theta_3(v,\tau)=\prod_{j=1}^{\infty}[(1-q^j)(1+e^{2\pi\sqrt{-1}v}q^{j-\frac{1}{2}})
(1+e^{-2\pi\sqrt{-1}v}q^{j-\frac{1}{2}})],
\end{equation}
 \noindent
where $q=e^{2\pi\sqrt{-1}\tau}$ with $\tau\in\textbf{H}$, the upper
half complex plane. Let
\begin{equation}
    \theta'(0,\tau)=\frac{\partial\theta(v,\tau)}{\partial v}|_{v=0}.
\end{equation} \noindent Then the following Jacobi identity
(cf. \cite{Ch}) holds,
\begin{equation}   \theta'(0,\tau)=\pi\theta_1(0,\tau)\theta_2(0,\tau)\theta_3(0,\tau).
\end{equation}
\noindent Denote $$SL_2({\bf Z})=\left\{\left(\begin{array}{cc}
\ a & b  \\
 c  & d
\end{array}\right)\mid a,b,c,d \in {\bf Z},~ad-bc=1\right\}$$ the
modular group. Let $S=\left(\begin{array}{cc}
\ 0 & -1  \\
 1  & 0
\end{array}\right),~T=\left(\begin{array}{cc}
\ 1 &  1 \\
 0  & 1
\end{array}\right)$ be the two generators of $SL_2(\bf{Z})$. They
act on $\textbf{H}$ by $S\tau=-\frac{1}{\tau},~T\tau=\tau+1$. One
has the following transformation laws of theta functions under the
actions of $S$ and $T$ (cf. \cite{Ch}):
\begin{equation}
    \theta(v,\tau+1)=e^{\frac{\pi\sqrt{-1}}{4}}\theta(v,\tau),~~\theta(v,-\frac{1}{\tau})
=\frac{1}{\sqrt{-1}}\left(\frac{\tau}{\sqrt{-1}}\right)^{\frac{1}{2}}e^{\pi\sqrt{-1}\tau
v^2}\theta(\tau v,\tau);
\end{equation}
 \begin{equation}
     \theta_1(v,\tau+1)=e^{\frac{\pi\sqrt{-1}}{4}}\theta_1(v,\tau),~~\theta_1(v,-\frac{1}{\tau})
=\left(\frac{\tau}{\sqrt{-1}}\right)^{\frac{1}{2}}e^{\pi\sqrt{-1}\tau
v^2}\theta_2(\tau v,\tau);
 \end{equation}
\begin{equation}   \theta_2(v,\tau+1)=\theta_3(v,\tau),~~\theta_2(v,-\frac{1}{\tau})
=\left(\frac{\tau}{\sqrt{-1}}\right)^{\frac{1}{2}}e^{\pi\sqrt{-1}\tau
v^2}\theta_1(\tau v,\tau);
\end{equation}
\begin{equation}    \theta_3(v,\tau+1)=\theta_2(v,\tau),~~\theta_3(v,-\frac{1}{\tau})
=\left(\frac{\tau}{\sqrt{-1}}\right)^{\frac{1}{2}}e^{\pi\sqrt{-1}\tau
v^2}\theta_3(\tau v,\tau).
\end{equation}
\begin{equation}
 \theta'(0,-\frac{1}{\tau})=\frac{1}{\sqrt{-1}}\left(\frac{\tau}{\sqrt{-1}}\right)^{\frac{1}{2}}
\tau\theta'(0,\tau),
\end{equation}

\begin{equation}
    \theta'(v,\tau+1)=e^{\frac{\pi\sqrt{-1}}{4}}\theta'(v,\tau),
 \end{equation}
 \begin{equation}
    ~~\theta'(v,-\frac{1}{\tau})
=\frac{1}{\sqrt{-1}}\left(\frac{\tau}{\sqrt{-1}}\right)^{\frac{1}{2}}e^{\pi\sqrt{-1}\tau
v^2}(2\pi \sqrt{-1}\tau\theta(\tau v,\tau)+\tau\theta'(\tau v,\tau));
\end{equation}
 \begin{equation}
    \theta'_1(v,\tau+1)=e^{\frac{\pi\sqrt{-1}}{4}}\theta'_1(v,\tau),
 \end{equation}
 \begin{equation}
    ~~\theta'_1(v,-\frac{1}{\tau})
=\frac{1}{\sqrt{-1}}\left(\frac{\tau}{\sqrt{-1}}\right)^{\frac{1}{2}}e^{\pi\sqrt{-1}\tau
v^2}(2\pi \sqrt{-1}\tau\theta_2(\tau v,\tau)+\tau\theta'_2(\tau v,\tau));
\end{equation}
\begin{equation}
    \theta'_2(v,\tau+1)=e^{\frac{\pi\sqrt{-1}}{4}}\theta'_3(v,\tau),
 \end{equation}
 \begin{equation}
    ~~\theta'_2(v,-\frac{1}{\tau})
=\frac{1}{\sqrt{-1}}\left(\frac{\tau}{\sqrt{-1}}\right)^{\frac{1}{2}}e^{\pi\sqrt{-1}\tau
v^2}(2\pi \sqrt{-1}\tau\theta_1(\tau v,\tau)+\tau\theta'_1(\tau v,\tau));
\end{equation}
\begin{equation}
    \theta'_3(v,\tau+1)=e^{\frac{\pi\sqrt{-1}}{4}}\theta'_2(v,\tau),
 \end{equation}
 \begin{equation}
    ~~\theta'_3(v,-\frac{1}{\tau})
=\frac{1}{\sqrt{-1}}\left(\frac{\tau}{\sqrt{-1}}\right)^{\frac{1}{2}}e^{\pi\sqrt{-1}\tau
v^2}(2\pi \sqrt{-1}\tau\theta_3(\tau v,\tau)+\tau\theta'_3(\tau v,\tau));
\end{equation}
\noindent
 \noindent {\bf Definition 2.1} A modular form over $\Gamma$, a
 subgroup of $SL_2({\bf Z})$, is a holomorphic function $f(\tau)$ on
 $\textbf{H}$ such that
 \begin{equation}
    f(g\tau):=f\left(\frac{a\tau+b}{c\tau+d}\right)=\chi(g)(c\tau+d)^kf(\tau),
 ~~\forall g=\left(\begin{array}{cc}
\ a & b  \\
 c & d
\end{array}\right)\in\Gamma,
 \end{equation}
\noindent where $\chi:\Gamma\rightarrow {\bf C}^{\star}$ is a
character of $\Gamma$. $k$ is called the weight of $f$.\\
Let $$\Gamma_0(2)=\left\{\left(\begin{array}{cc}
\ a & b  \\
 c  & d
\end{array}\right)\in SL_2({\bf Z})\mid c\equiv 0~({\rm
mod}~2)\right\},$$
$$\Gamma^0(2)=\left\{\left(\begin{array}{cc}
\ a & b  \\
 c  & d
\end{array}\right)\in SL_2({\bf Z})\mid b\equiv 0~({\rm
mod}~2)\right\},$$
be the two modular subgroups of $SL_2({\bf Z})$.
It is known that the generators of $\Gamma_0(2)$ are $T,~ST^2ST$,
the generators of $\Gamma^0(2)$ are $STS,~T^2STS$ (cf. \cite{Ch}).\\
\indent If $\Gamma$ is a modular subgroup, let ${\mathcal{M}}_{{\bf
R}}(\Gamma)$ denote the ring of modular forms over $\Gamma$ with
real Fourier coefficients. Writing $\theta_j=\theta_j(0,\tau),~1\leq
j\leq 3,$ we introduce six explicit modular forms (cf. \cite{Li1}),
$$\delta_1(\tau)=\frac{1}{8}(\theta_2^4+\theta_3^4),~~\varepsilon_1(\tau)=\frac{1}{16}\theta_2^4\theta_3^4,$$
$$8\delta_2(\tau)=-\frac{1}{8}(\theta_1^4+\theta_3^4),~~\varepsilon_2(\tau)=\frac{1}{16}\theta_1^4\theta_3^4,$$
\noindent They have the following Fourier expansions in
$q^{\frac{1}{2}}$:
$$\delta_1(\tau)=\frac{1}{4}+6q+\cdots,~~\varepsilon_1(\tau)=\frac{1}{16}-q+\cdots,$$
$$8\delta_2(\tau)=-1-24q^{\frac{1}{2}}-24q+\cdots,~~\varepsilon_2(\tau)=q^{\frac{1}{2}}+8q+\cdots,$$
\noindent where the $"\cdots"$ terms are the higher degree terms,
all of which have integral coefficients. They also satisfy the
transformation laws,
\begin{equation}
    \delta_2(-\frac{1}{\tau})=\tau^2\delta_1(\tau),~~~~~~\varepsilon_2(-\frac{1}{\tau})
=\tau^4\varepsilon_1(\tau),
\end{equation}
\noindent {\bf Lemma 2.2} (\cite{Li1}) {\it $\delta_1(\tau)$ (resp.
$\varepsilon_1(\tau)$) is a modular form of weight $2$ (resp. $4$)
over $\Gamma_0(2)$, $\delta_2(\tau)$ (resp. $\varepsilon_2(\tau)$)
is a modular form of weight $2$ (resp. $4$) over $\Gamma^0(2)$,
while  $\delta_3(\tau)$ (resp. $\varepsilon_3(\tau)$) is a modular
form of weight $2$ (resp. $4$) over $\Gamma_\theta(2)$ and moreover
${\mathcal{M}}_{{\bf R}}(\Gamma^0(2))={\bf
R}[\delta_2(\tau),\varepsilon_2(\tau)]$.}\\

We recall the Eisenstein series $G_{2k}(\tau)$ are defined to be
  \begin{equation}
  E_{2k}(\tau):=1-\frac{4k}{B_{2k}}\sum_{n=1}^{\infty}\sigma_{2k-1}(n)\cdot q^n,~~G_{2k}(\tau)=-\frac{B_{2k}}{4k}\cdot E_{2k}(\tau)
\end{equation}
where $\sigma_k(n):=\sum_{m>0,~m|n}m^k$ and $B_{2k}$ are the Bernoulli numbers. It is well known that the whole grading ring of modular forms
over $SL(2,Z)$ are generated by $E_4(\tau)$ and $E_6(\tau)$ and
 \begin{equation}
  E_{4}(\tau):=1+240q+2160q^2+\cdots,~~E_{6}(\tau)=1-504q-16632q^2+\cdots.
\end{equation}
\\
\noindent{\bf Definition 2.3} A meromorphic Jacobi form of index $m$ and weight $l$ over $L\times \Gamma$, where $L$ is an integral lattice in the complex plane $\mathbb{C}$ preserved by the modular subgroup $\Gamma\subset SL(2,Z)$, is a (meromorphic) function $F(z,\tau)$ on $C\times H$ such that
\begin{align}
F(\frac{z}{c\tau+d_0},\frac{a\tau+b}{c\tau+d_0})=(c\tau+d_0)^{l}{\rm exp}(2\pi\sqrt{-1}m\frac{cz^2}{c\tau+d_0}){F}(z,\tau),
\end{align}
\begin{align}
{F}(z+\lambda \tau+\mu,\tau)=(-1)^{m(\mu +\lambda) }{\rm exp}(-2\pi\sqrt{-1}m(2\lambda z+\lambda^2\tau))
{F}(z,\tau),
\end{align}
where $\left(\begin{array}{cc}
\ a & b  \\
 c  & d_0
\end{array}\right)\in \Gamma$ and $\lambda,\mu \in L$.\\

\section{The two variable elliptic genus in odd dimensions}

\indent Let $M$ be a $(2d+1)$-dimensional closed ${\rm spin}$ manifold. Let $W$ denote a complex $l$-dimensional vector bundle on $M$. Denote the first Chern classes of $W$ by $c_1(W)$. We denote by
$\pm 2\pi \sqrt{-1}x_i~(1\leq i\leq d),~0$ and $2\pi \sqrt{-1}w_j~(1\leq j\leq l)$ respectively the formal Chern roots of $TM\otimes C$ and $W$. Formally the Todd form of $M$ is written by (probably not well defined)
\begin{equation}
  {\rm Td}(M):=\prod_{i=1}^d\frac{2\pi \sqrt{-1}x_i}{1-e^{-2\pi \sqrt{-1}x_i}}.
\end{equation}
Let $(\tau,z)\in \mathcal{H}\times \mathcal{C}$ where $\mathcal{H}$ is the upper half plane and $\mathcal{C}$ is the complex plane. Let $y=e^{2\pi \sqrt{-1}z}$ and $q= e^{2\pi \sqrt{-1}\tau}$
and $c:=\prod_{j=1}^{\infty}(1-q^j)$.\\

Let \begin{align}
{\rm E}(M,W,\tau,z):=&c^{2(d-l)+1}y^{-\frac{l}{2}} \bigotimes _{n=1}^{\infty}
\left( \wedge_{-yq^{n-1}}(W^*)\otimes\wedge_{-y^{-1}q^{n}}(W)\right)\\\notag
&\bigotimes\left(\bigotimes _{n=1}^{\infty}S_{q^n}(TM\otimes C)\right).
\end{align}

Now let $g:M\to SO(N)$ and we assume that $N$ is even and large enough. Let $E$ denote the trivial real vector bundle of rank $N$ over $M$. We equip $E$ with the canonical trivial metric and trivial connection $d$. Set
$$\nabla_u=d+ug^{-1}dg,\ \ u\in[0,1].$$
Let $R_u$ be the curvature of $\nabla_u$, then
\begin{equation}
  R_u=(u^2-u)(g^{-1}dg)^2.
\end{equation}
We also consider the complexification of $E$ and $g$ extends to a unitary automorphism of $E_{\mathbf{C}}$. The connection $\nabla_u$ extends to a Hermitian connection on $E_{\mathbf{C}}$ with curvature still given by (3.3). Let $\Delta(E)$ be the spinor bundle of $E$, which is a trivial Hermitian
bundle of rank $2^{\frac{N}{2}}$. We assume that $g$ has a lift to the Spin group ${\rm Spin}(N):g^{\Delta}:M\to {\rm Spin}(N)$. So $g^{\Delta}$ can be viewed as an automorphism of $\Delta(E)$ preserving the Hermitian metric. We lift $d$ on $E$ to be a trivial Hermitian connection $d^{\Delta}$ on $\Delta(E)$, then
\begin{equation}
  \nabla_u^{\Delta}=(1-u)d^{\Delta}+u(g^{\Delta})^{-1}\cdot d^{\Delta} \cdot g^{\Delta},\ \ u\in[0,1]
\end{equation}
lifts $\nabla_u$ on $E$ to $\Delta(E)$. Let $Q_j(E),j=1,2,3$ be the virtual bundles defined as follows:
\begin{equation}
Q_1(E)=\triangle(E)\otimes
   \bigotimes _{n=1}^{\infty}\wedge_{q^n}(\widetilde{E_C});
\end{equation}
\begin{equation}
Q_2(E)=\bigotimes _{n=1}^{\infty}\wedge_{-q^{n-\frac{1}{2}}}(\widetilde{E_C});
\end{equation}
\begin{equation}
Q_3(E)=\bigotimes _{n=1}^{\infty}\wedge_{q^{n-\frac{1}{2}}}(\widetilde{E_C}).
\end{equation}
Let $g$ on $E$ have a lift $g^{Q(E)}$ on $Q(E)$ and $\nabla_u$ have a lift $\nabla^{Q(E)}_u$ on $Q(E)$. Following \cite{HY}, we defined ${\rm ch}(Q(E),g^{Q(E)},d,\tau)$ as following
\begin{equation}
{\rm ch}(Q(E),\nabla^{Q(E)}_0,\tau)-{\rm ch}(Q(E),\nabla^{Q(E)}_1,\tau)=-d{\rm ch}(Q(E),g^{Q(E)},d,\tau),
\end{equation}
where
$$Q(E)=Q_1(E)\otimes Q_2(E)\otimes Q_3(E),$$
and
\begin{equation}
{\rm ch}(Q(E),g^{Q(E)},d,\tau)=-\frac{2^{\frac{N}{2}}}{8\pi^2}\int^1_0{\rm Tr}\left[g^{-1}dg\left(A\right)\right]du,
\end{equation}
with
$$A=\frac{\theta'_1(R_u/(4\pi^2),\tau)}{\theta_1(R_u/(4\pi^2),\tau)}
+\frac{\theta'_2(R_u/(4\pi^2),\tau)}{\theta_2(R_u/(4\pi^2),\tau)}+\frac{\theta'_3(R_u/(4\pi^2),\tau)}
{\theta_3(R_u/(4\pi^2),\tau)}.$$
By Proposition 2.2 in \cite{HY}, we have if $c_3(E_C,g,d)=0$, then for any integer $r\geq 1.$ We have ${\rm ch}(Q(E),g^{Q(E)},d,\tau+1)^{(4r-1)}={\rm ch}(Q(E),g^{Q(E)},d,\tau)^{(4r-1)}$and ${\rm ch}(Q(E),g^{Q(E)},d,-\frac{1}{\tau})^{(4r-1)}=\tau^{2r}{\rm ch}(Q(E),g^{Q(E)},d,\tau)^{(4r-1)}$,so ${\rm ch}(Q(E),g^{Q(E)},d,\tau)^{(4r-1)}$ are modular forms of weight $2r$ over $SL_2(\bf{Z})$.
\begin{defn}
The two variable elliptic genus of $M^{2d+1}$ with respect to $W$, which we denote by ${\rm Ell}(M,W,g,\tau,z)$ is defined by
\begin{equation}
{\rm Ell}(M,W,g,\tau,z):=\int_M \widehat{A}(M){\rm ch}(E(M,W,\tau,z)){\rm ch}(Q(E),g^{Q(E)},d,\tau),
\end{equation}
which is the index of the Topelitz operator associated to the Dirac operator twisted by $E(M,W,\tau,z)$ and $g^{Q(E)}$.
\end{defn}

Using the same calculations as in Lemma 3.4 in \cite{LiP} and
\begin{equation}
  {\widehat{A}}(M)=\prod_{i=1}^d\frac{\pi \sqrt{-1}x_i}{{\rm sinh }{\pi \sqrt{-1}x_i}}=\prod_{i=1}^d\frac{2\pi \sqrt{-1}x_i}{1-e^{-2\pi \sqrt{-1}x_i}}
  e^{-\sum_{i=1}^n\pi \sqrt{-1}x_i},
\end{equation}
 we have

\begin{lem}
\begin{align}
{\rm Ell}(M,W,g,\tau,z)=\int_Me^{-\frac{c_1(W)}{2}}\eta(\tau)^{3(d-l)}\prod_{i=1}^d\frac{2\pi \sqrt{-1}x_i}{\theta(\tau,x_i)}\prod_{j=1}^l\theta
(\tau,w_j-z){\rm ch}(Q(E),g^{Q(E)},d,\tau),
\end{align}
where
 \begin{equation}
 \eta(\tau):=q^{\frac{1}{24}}\cdot c=q^{\frac{1}{24}}\prod_{j=1}^{\infty}(1-q^j).
   \end{equation}
\end{lem}
Similarly to Theorem 3.2 in \cite{LiP}, we have
\begin{thm}
If $c_1(W)=0$ and the first Pontrjagin classes $p_1(M)=p_1(W)$ and the third cohomology group $H^3(M,R)=0$ and $M$ is simple connected spin manifold, then
the elliptic genus ${\rm Ell}(M,W,g,\tau,z)$ is a weak Jacobi form of weight $(d+1)-l$ and index $\frac{l}{2}$ which satisfied\\
{\rm (1)}${\rm Ell}(N,W,g,\tau,z)=0$ when $N$ is a boundary;\\
{\rm (2)} ${\rm Ell}(M_1\amalg M_2,W,g,\tau,z)={\rm Ell}(M_1,W,g,\tau,z)+{\rm Ell}(M_2,W,g,\tau,z)$;\\
{\rm (3)} ${\rm Ell}(M_1^{2n}\times M_2^{2m+1},g,\tau,z)={\rm Ell}(M_1^{2n},\tau,z){\rm Ell}(M_2^{2m+1},g,\tau,z)$ where ${\rm Ell}(M_1^{2n},\tau,z)$
is the two variable elliptic genus for even dimensional manifolds and $g:M_2\rightarrow GL(N,C).$
\end{thm}

We recall Proposition 3.5 in \cite{LiP}.
\begin{prop}(\cite{LiP}) Suppose a function $\varphi(\tau,z):\mathbb{H}\times \mathbb{C}\rightarrow \mathbb{C}$ satisfies
\begin{align}
\varphi(\frac{a\tau+b}{c\tau+d_0},\frac{z}{c\tau+d_0})=(c\tau+d_0)^{k}{\rm exp}(\frac{2\pi\sqrt{-1}mcz^2}{c\tau+d_0})\varphi(\tau,z);~~
\left(\begin{array}{cc}
\ a & b  \\
 c  & d_0
\end{array}\right)\in SL(2,Z).
\end{align}
We define
\begin{align}
&\Phi(\tau,z):={\rm exp}(-8\pi^2mG_2(\tau)z^2)\varphi(\tau,z):=\sum_{n\geq 0}a_n(\tau)\cdot z^n,
\end{align}
then these $a_n(\tau)$ are modular forms of weight $k+n$ over $SL(2,Z)$.
\end{prop}

\begin{prop}If $c_1(W)=0$ and the first Pontrjagin classes $p_1(M)=p_1(W)$ and the third cohomology group $H^3(M,R)=0$ and $M$ is simple connected spin manifold, then the series $a_n(M,W,g,\tau)$ determined by
\begin{align}
{\rm exp}(-4\pi^2lG_2(\tau)z^2){\rm Ell}(M,W,g,\tau,z)=\sum_{n\geq 0}a_n(M,W,g,\tau)\cdot z^n,
\end{align}
are modular forms of weight $d+1-l+n$ over $SL(2,Z)$. Furthermore, the first five series of $a_n(M,W,g,\tau)$ are of the following form:
\begin{align}
&a_0(M,W,g,\tau)=\int_M\widehat{A}(TM){\rm ch}(\wedge_{-1}(W^*)){\rm ch}(\triangle(E),g,d)\\\notag
&+q
\left\{\int_M\widehat{A}(TM){\rm ch}(\wedge_{-1}(W^*)\otimes[-2(d-l)-1-W-W^*+TM\otimes C]){\rm ch}(\triangle(E),g,d)\right.\\\notag
&\left.+\int_M\widehat{A}(TM){\rm ch}(\wedge_{-1}(W^*)){\rm ch}(\triangle(E)\otimes(\widetilde{E_C}+2\wedge^2\widetilde{E_C}
-\widetilde{E_C}\otimes \widetilde{E_C}),g,d)\right\}
+O(q^2),
\end{align}

\begin{align}
&a_1(M,W,g,\tau)=\sum_{p=0}^l(-1)^p2\pi\sqrt{-1}(p-\frac{l}{2})\int_M\widehat{A}(TM){\rm ch}(\wedge^p(W^*)){\rm ch}(\triangle(E),g,d)
+a^1_1q+O(q^2),
\end{align}
where
\begin{align}
&a^1_1=-2\pi\sqrt{-1}\int_M\widehat{A}(TM){\rm ch}(\wedge_{-1}(W^*)\otimes W\otimes W^*){\rm ch}(\triangle(E),g,d)\\\notag
&+\pi\sqrt{-1}\int_M\widehat{A}(TM){\rm ch}([2\sum_{j=1}^l(-1)^jj\wedge^jW^*-l\wedge_{-1}W^*]\\\notag
&\otimes[-2(d-l)-1-W-W^*+TM\otimes C])
{\rm ch}(\triangle(E),g,d)\\\notag
&+\pi\sqrt{-1}\int_M\widehat{A}(TM){\rm ch}([2\sum_{j=1}^l(-1)^jj\wedge^jW^*-l\wedge_{-1}W^*])\\\notag
&\cdot{\rm ch}(\triangle(E)\otimes(\widetilde{E_C}+2\wedge^2\widetilde{E_C}
-\widetilde{E_C}\otimes \widetilde{E_C}),g,d).
\end{align}

\begin{align}
&a_2(M,W,g,\tau)=-\sum_{p=0}^l(-1)^p2\pi^2(p-\frac{l}{2})^2\int_M\widehat{A}(TM){\rm ch}(\wedge^p(W^*)){\rm ch}(\triangle(E),g,d)\\\notag
&+\sum_{p=0}^l(-1)^p\frac{l}{6}\pi^2\int_M\widehat{A}(TM){\rm ch}(\wedge^p(W^*)){\rm ch}(\triangle(E),g,d)+O(q),
\end{align}

\begin{align}
&a_3(M,W,g,\tau)=\sum_{p=0}^l(-1)^p\frac{1}{6}(2\pi\sqrt{-1})^3(p-\frac{l}{2})^3\int_M\widehat{A}(TM){\rm ch}(\wedge^p(W^*)){\rm ch}(\triangle(E),g,d)\\\notag
&-\sum_{p=0}^l(-1)^pl\frac{2}{3}\pi^4(p-\frac{l}{2})\int_M\widehat{A}(TM){\rm ch}(\wedge^p(W^*)){\rm ch}(\triangle(E),g,d)+O(q),
\end{align}

\begin{align}
&a_4(M,W,g,\tau)=\sum_{p=0}^l(-1)^p\frac{1}{24}(2\pi\sqrt{-1})^4(p-\frac{l}{2})^4\int_M\widehat{A}(TM){\rm ch}(\wedge^p(W^*)){\rm ch}(\triangle(E),g,d)\\\notag
&-\sum_{p=0}^l(-1)^p\frac{l}{3}\pi^4(p-\frac{l}{2})^2\int_M\widehat{A}(TM){\rm ch}(\wedge^p(W^*)){\rm ch}(\triangle(E),g,d)\\\notag
&+\sum_{p=0}^l(-1)^p\frac{\pi^4 l^2}{12}\int_M\widehat{A}(TM){\rm ch}(\wedge^p(W^*)){\rm ch}(\triangle(E),g,d)+O(q),
\end{align}
\end{prop}
\begin{proof} We know that ${\rm Ell}(M,W,g,\tau,0)=a_0(M,W,g,\tau)$ and
\begin{equation}{\rm Ell}(M,W,g,\tau,0)=\int_M\widehat{A}(TM){\rm ch}(E(M,W,g,\tau,0)){\rm ch}(Q(E),g^{Q(E)},d),
\end{equation}
where
\begin{equation}
{\rm E}(M,W,\tau,0):=c^{2(d-l)+1}\wedge_{-1}(W^*)\bigotimes _{n=1}^{\infty}
\wedge_{-q^n}(W^*)\bigotimes\wedge_{-q^{n}}(W)
\bigotimes\left(\bigotimes _{n=1}^{\infty}S_{q^n}(TM\otimes C)\right).
\end{equation}
So we get (3.17). We know that
\begin{equation}
{\rm ch}(Q(E),g^{Q(E)},d)={\rm ch}(\triangle(E),g,d)+{\rm ch}(\triangle(E)\otimes(\widetilde{E_C}+2\wedge^2\widetilde{E_C}
-\widetilde{E_C}\otimes \widetilde{E_C}),g,d)q+O(q^2).
\end{equation}
Similar to Proposition 3.2 in \cite{Wa1}, then we can get Proposition 3.5 by (3.23) and (3.25).
\end{proof}

Since there are no $SL(2,Z)$ modular forms with the odd weight or the non zero weight $\leq 2$, we have
\begin{prop}
If $c_1(W)=0$ and the first Pontrjagin classes $p_1(M)=p_1(W)$ and the third cohomology group $H^3(M,R)=0$ and $M$ is a simple connected spin $(2d+1)$-manifold, then\\
1)if either $d+1-l$ is odd or $d+1-l\leq 2$ but $d+1-l\neq 0$, then
\begin{align}
&a^0_0=\int_M\widehat{A}(TM){\rm ch}(\wedge_{-1}(W^*)){\rm ch}(\triangle(E),g,d)=0\\\notag
&
a^1_0=\int_M\widehat{A}(TM){\rm ch}(\wedge_{-1}(W^*)\otimes[-2(d-l)-1-W-W^*+TM\otimes C]){\rm ch}(\triangle(E),g,d)\\\notag
&+\int_M\widehat{A}(TM){\rm ch}(\wedge_{-1}(W^*)){\rm ch}(\triangle(E)\otimes(\widetilde{E_C}+2\wedge^2\widetilde{E_C}
-\widetilde{E_C}\otimes \widetilde{E_C}),g,d)=0.
\end{align}
and the indices of associated Toeplitz operators are zero. \\
2)if either $d+1-l$ is even or $d+1-l\leq 1$ but $d+1-l\neq -1$, then
\begin{align}
a^0_1=\sum_{p=0}^l(-1)^p2\pi\sqrt{-1}(p-\frac{l}{2})\int_M\widehat{A}(TM){\rm ch}(\wedge^p(W^*)){\rm ch}(\triangle(E),g,d)=0,~~~~a^1_1=0.
\end{align}

3)if either $d+1-l$ is odd or $d+1-l\leq 0$ but $d+1-l\neq -2$, then
\begin{align}
&-\sum_{p=0}^l(-1)^p2\pi^2(p-\frac{l}{2})^2\int_M\widehat{A}(TM){\rm ch}(\wedge^p(W^*)){\rm ch}(\triangle(E),g,d)\\\notag
&+\sum_{p=0}^l(-1)^p\frac{l}{6}\pi^2\int_M\widehat{A}(TM){\rm ch}(\wedge^p(W^*)){\rm ch}(\triangle(E),g,d)=0.
\end{align}
4)if either $d+1-l$ is even or $d+1-l\leq -1$ but $d+1-l\neq -3$, then
\begin{align}
&\sum_{p=0}^l(-1)^p\frac{1}{6}(2\pi\sqrt{-1})^3(p-\frac{l}{2})^3\int_M\widehat{A}(TM){\rm ch}(\wedge^p(W^*)){\rm ch}(\triangle(E),g,d)\\\notag
&-\sum_{p=0}^l(-1)^pl\frac{2}{3}\pi^4(p-\frac{l}{2})\int_M\widehat{A}(TM){\rm ch}(\wedge^p(W^*)){\rm ch}(\triangle(E),g,d)=0.
\end{align}
5)if either $d+1-l$ is odd or $d+1-l\leq -2$ but $d+1-l\neq -4$, then
\begin{align}
&\sum_{p=0}^l(-1)^p\frac{1}{24}(2\pi\sqrt{-1})^4(p-\frac{l}{2})^4\int_M\widehat{A}(TM){\rm ch}(\wedge^p(W^*)){\rm ch}(\triangle(E),g,d)\\\notag
&-\sum_{p=0}^l(-1)^p\frac{l}{3}\pi^4(p-\frac{l}{2})^2\int_M\widehat{A}(TM){\rm ch}(\wedge^p(W^*)){\rm ch}(\triangle(E),g,d)\\\notag
&+\sum_{p=0}^l(-1)^p\frac{\pi^4 l^2}{12}\int_M\widehat{A}(TM){\rm ch}(\wedge^p(W^*)){\rm ch}(\triangle(E),g,d)=0.
\end{align}
Similarly, we have expressions using the twisted Toeplitz operators indices.
\end{prop}
\begin{thm}If $c_1(W)=0$ and the first Pontrjagin classes $p_1(M)=p_1(W)$ and the third cohomology group $H^3(M,R)=0$ and $M$ is a simple connected spin $(2d+1)$-manifold, then\\
1)if $d+1-l=4$, then $a_0^1=240a^0_0$ and $a^1_0$ is the integer multiple of $240$.\\
2)if $d+1-l=6$, then $a_0^1=-504a^0_0$ and $a^1_0$ is the integer multiple of $504$.\\
3)if $d+1-l=8$, then $a_0^1=480a^0_0$ and $a^1_0$ is the integer multiple of $480$.\\
4)if $d+1-l=10$, then $a_0^1=-264a^0_0$ and $a^1_0$ is the integer multiple of $264$.\\
5)if $d+1-l=3$, then $a_1^1=240a^0_1$ and $a^1_1$ is the integer multiple of $240$.\\
6)if $d+1-l=5$, then $a_1^1=-504a^0_1$ and $a^1_1$ is the integer multiple of $504$.\\
7)if $d+1-l=7$, then $a_1^1=480a^0_1$ and $a^1_1$ is the integer multiple of $480$.\\
8)if $d+1-l=9$, then $a_1^1=-264a^0_1$ and $a^1_1$ is the integer multiple of $264$.\\
\end{thm}
\begin{proof}
By (2.28) and Proposition 3.5, we can prove Theorem 3.7.
\end{proof}

\section{Two-variable elliptic genera which are $\Gamma_0(2)$, $\Gamma^0(2)$, $\Gamma _{\theta}$-Jacobi forms}
\subsection{Even dimensional case}
Let $V$ be $(2r)$-dimensional spin real vector bundle on the almost complex $(2d)$-dimensional manifold $M$ and $\triangle(V)$ be the associated spinors bundle. We denote by
$\pm 2\pi \sqrt{-1}v_s~(1\leq s\leq r)$ the formal Chern roots of $V\otimes C$.
Let $Q_j(V)$, $j=1,2,3$ be the virtual bundles defined as following:
\begin{equation}
 Q_1(V)=\triangle(V)\otimes
   \bigotimes _{n=1}^{\infty}\wedge_{q^n}(\widetilde{V_C}),
Q_2(V)=\bigotimes _{n=1}^{\infty}\wedge_{-q^{n-\frac{1}{2}}}(\widetilde{V_C});~~
Q_3(V)=\bigotimes _{n=1}^{\infty}\wedge_{q^{n-\frac{1}{2}}}(\widetilde{V_C}).
\end{equation}
\begin{defn}
The elliptic genus of $(M^{2d},J)$ with respect to $W$ and $V$, which we denote by ${\rm Ell}_a(M,W,V,\tau,z)$ for $1\leq a\leq 3$ is defined by
\begin{equation}
{\rm Ell}_1(M,W,V,\tau,z):=\int_MTd(M){\rm ch}(E'(M,W,\tau,z)\otimes Q_1(V)),
\end{equation}
\begin{equation}
{\rm Ell}_2(M,W,V,\tau,z):=2^r\int_MTd(M){\rm ch}(E'(M,W,\tau,z)\otimes Q_2(V)),
\end{equation}
\begin{equation}
{\rm Ell}_3(M,W,V,\tau,z):=2^r\int_MTd(M){\rm ch}(E'(M,W,\tau,z)\otimes Q_3(V)),
\end{equation}
where
\begin{align}
{\rm E'}(M,W,\tau,z):=&c^{2(d-l)}y^{-\frac{l}{2}} \bigotimes _{n=1}^{\infty}
\left( \wedge_{-yq^{n-1}}(W^*)\wedge_{-y^{-1}q^{n}}(W)\right)\\\notag
&\bigotimes\left(\bigotimes _{n=1}^{\infty}S_{q^n}(TM\otimes C)\right).
\end{align}
\end{defn}
Similar to Lemma 3.4 in \cite{LiP}, we have

\begin{lem}The following equality holds
\begin{align}
&{\rm Ell}_a(M,W,V,\tau,z)=2^r\int_M{\rm exp}(\frac{c_1(M)-c_1(W)}{2})\eta(\tau)^{3(d-l)}\\\notag
&\cdot\prod_{i=1}^d\frac{2\pi \sqrt{-1}x_i}{\theta(\tau,x_i)}\prod_{j=1}^l\theta
(\tau,w_j-z)\prod_{s=1}^{r}\frac{\theta_a(w_s,\tau)}{\theta_a(0,\tau)}, ~~1\leq a \leq 3.
\end{align}
\end{lem}

\begin{thm}
If $c_1(M)=c_1(W)=0$ and $p_1(V)=0$ and the first Pontrjagin classes $p_1(M)=p_1(W)$, then the elliptic genus ${\rm Ell}_1(M,W,V,\tau,z)$,
${\rm Ell}_2(M,W,V,\tau,z)$, ${\rm Ell}_3(M,W,V,\tau,z)$ are respectively weak $\Gamma_0(2)$, $\Gamma^0(2)$, $\Gamma _{\theta}$-Jacobi forms
of weight $d-l$ and index $\frac{l}{2}$.
\end{thm}
\begin{proof}
By (2.12)-(2.24), we have
we get ${\rm Ell}_a(M,W,V,\tau,z)$ satisfies the following transformation laws:
\begin{align}
&{\rm Ell}_1(M,W,\tau+1,z)={\rm Ell}_1(M,W,V,\tau,z),\\\notag
&{\rm Ell}_2(M,W,\tau+1,z)={\rm Ell}_3(M,W,V,\tau,z),\\\notag
&{\rm Ell}_3(M,W,\tau+1,z)={\rm Ell}_2(M,W,V,\tau,z),\\\notag
&{\rm Ell}_a(M,W,V,\tau,z+1)=(-1)^{l}{\rm Ell}_a(M,W,V,\tau,z),\\\notag
&{\rm Ell}_a(M,W,V,\tau,z+\tau)=(-1)^{l}{\rm exp}(-\pi\sqrt{-1}l(\tau+2z))
{\rm Ell}_a(M,W,V,\tau,z),\\\notag
&{\rm Ell}_1(M,W,V,-\frac{1}{\tau},\frac{z}{\tau})=\tau^{d-l}{\rm exp}(\pi\sqrt{-1}l\frac{z^2}{\tau})
{\rm Ell}_2(M,W,V,\tau,z),\\\notag
&{\rm Ell}_2(M,W,V,-\frac{1}{\tau},\frac{z}{\tau})=\tau^{d-l}{\rm exp}(\pi\sqrt{-1}l\frac{z^2}{\tau})
{\rm Ell}_1(M,W,V,\tau,z),\\\notag
&{\rm Ell}_3(M,W,V,-\frac{1}{\tau},\frac{z}{\tau})=\tau^{d-l}{\rm exp}(\pi\sqrt{-1}l\frac{z^2}{\tau})
{\rm Ell}_3(M,W,V,\tau,z).
\end{align}
By (4.7), we can prove Theorem 4.3.
\end{proof}
Similar to Proposition 3.5, we have
\begin{prop}
If $c_1(M)=c_1(W)=0$ and $p_1(V)=0$ and the first Pontrjagin classes $p_1(M)=p_1(W)$,  then the series $a_{n,1}(M,W,V,\tau)$,$a_{n,2}(M,W,V,\tau)$, $a_{n,3}(M,W,V,\tau)$ determined by
\begin{align}
{\rm exp}(-4\pi^2lG_2(\tau)z^2){\rm Ell}_\alpha(M,W,\tau,z)=\sum_{n\geq 0}a_{n,\alpha}(M,W,V,\tau)\cdot z^n,~~1\leq \alpha\leq 3
\end{align}
are modular forms of weight $d-l+n$ over $\Gamma_0(2)$, $\Gamma^0(2)$, $\Gamma _{\theta}$ respectively.
Let \begin{equation} a_{n,2}(M,W,V,\tau)=a_{n,2}^0+a_{n,2}^{\frac{1}{2}}q^{\frac{1}{2}}+a_{n,2}^1q+a_{n,2}^{\frac{3}{2}}q^{\frac{3}{2}}+\cdots+a_{n,2}^m+\cdots.
\end{equation}
Then $a_{n,2}^0$ is the same as $a_n^0$ in \cite{LiP} or \cite{Wa1}, and
\begin{align}
&a_{0,2}^{\frac{1}{2}}=-\int_MTd(M)\sum_{p=0}^l(-1)^p{\rm ch}(\wedge^p(W^*)\otimes\widetilde{V_C}),\\\notag
&a_{1,2}^{\frac{1}{2}}=-\int_MTd(M)\sum_{p=0}^l(-1)^p{\rm ch}(\wedge^p(W^*)\otimes\widetilde{V_C})(p-\frac{l}{2})2\pi\sqrt{-1},\\\notag
&a_{2,2}^{\frac{1}{2}}=\int_MTd(M)\sum_{p=0}^l(-1)^p{\rm ch}(\wedge^p(W^*)\otimes\widetilde{V_C})(p-\frac{l}{2})^22\pi^2\\\notag
&-\int_MTd(M)\sum_{p=0}^l(-1)^p{\rm ch}(\wedge^p(W^*)\otimes\widetilde{V_C})\frac{l}{6}\pi^2\\\notag
&a_{3,2}^{\frac{1}{2}}=-\int_MTd(M)\sum_{p=0}^l(-1)^p{\rm ch}(\wedge^p(W^*)\otimes\widetilde{V_C})(p-\frac{l}{2})^3\frac{4}{3}(\pi\sqrt{-1})^3\\\notag
&-\int_MTd(M)\sum_{p=0}^l(-1)^p{\rm ch}(\wedge^p(W^*)\otimes\widetilde{V_C})\frac{l}{3}\pi^3(p-\frac{l}{2})\sqrt{-1}£¬\\\notag
&a_{4,2}^{\frac{1}{2}}=-\int_MTd(M)\sum_{p=0}^l(-1)^p{\rm ch}(\wedge^p(W^*)\otimes\widetilde{V_C})(p-\frac{l}{2})^4\frac{2}{3}\pi^4\\\notag
&+\int_MTd(M)\sum_{p=0}^l(-1)^p{\rm ch}(\wedge^p(W^*)\otimes\widetilde{V_C})\frac{l^2}{3}\pi^4(p-\frac{l}{2})^2\\\notag
&-\frac{\pi^4l^2}{72}\int_MTd(M)\sum_{p=0}^l(-1)^p{\rm ch}(\wedge^p(W^*)\otimes\widetilde{V_C}).
\end{align}
\begin{align}
&a_{1,2}^{1}=\int_MTd(M)\sum_{p=0}^l(-1)^p{\rm ch}(\wedge^p(W^*)\otimes\wedge^2\widetilde{V_C})(p-\frac{l}{2})2\pi\sqrt{-1},\\\notag
&+\int_MTd(M){\rm ch}\left\{[-2(d-l)+T+T^*-(W+W^*)]2\pi\sqrt{-1}\otimes \sum_{j=1}^l(-1)^jj\wedge^jW^*\right.\\\notag
&\left.+(1+\wedge_{-1}W^*)[-2\pi\sqrt{-1}(W+W^*)-\pi\sqrt{-1}l[-2(d-l)+T+T^*-(W+W^*)]]\right\}.
\end{align}
\end{prop}

\begin{proof}
 If we set
\begin{equation}{\rm exp}(-4\pi^2lG_2(\tau)z^2):=A_0(Z)+A_1(z)q+O(q^2),
\end{equation}
and
\begin{equation}{\rm Ell}(M,W,\tau,z):=\widetilde{B_0}(Z)+\widetilde{B_1(z)}q+O(q^2),
\end{equation}
and
\begin{equation}{\rm Ell}_2(M,W,V,\tau,z):={B_0}(Z)+B_{\frac{1}{2}}(z)q^{\frac{1}{2}}+B_1(z)q+O(q^{\frac{3}{2}}),
\end{equation}
By
\begin{equation}Q_2(V)=1-\widetilde{V_C}q^{\frac{1}{2}}+\wedge^2\widetilde{V_C}q+O(q^{\frac{3}{2}}),
\end{equation}
then
\begin{equation}B_0(z)=\widetilde{B_0}(Z),~~{B_{\frac{1}{2}}}(z)=-\widetilde{B_0(z)}\otimes\widetilde{V_C},~~{B_1(z)}=\widetilde{B_0(z)}\otimes
\wedge^2\widetilde{V_C}+\widetilde{B_1(z)},
\end{equation}
we can get that
\begin{align}
&A_0(z)=1+\frac{l}{6}\pi^2z^2+\frac{l^2}{72}\pi^4z^4+O(z^6),\\\notag
&A_1(z)=-4l\pi^2z^2-\frac{2l^2}{3}\pi^4z^4+O(z^6),\\\notag
&\widetilde{B_0}(z)=\int_MTd(M)\sum_{p=0}^l(-1)^p{\rm ch}(\wedge^p(W^*))
\left[1+(p-\frac{l}{2})2\pi\sqrt{-1}z\right.\\\notag
&\left.+\frac{1}{2}(p-\frac{l}{2})^2(2\pi\sqrt{-1})^2z^2
+\frac{1}{6}(p-\frac{l}{2})^3(2\pi\sqrt{-1})^3z^3+O(z^4)\right],
\\\notag
&\widetilde{B}_1(z)=
\int_MTd(M){\rm ch}(\wedge_{-1}(W^*)\otimes[-2(d-l)+T+T^*-(W+W^*)]\\\notag
&-2\pi\sqrt{-1}\wedge_{-1}(W^*)\otimes(W+W^*)\\\notag
&+\pi\sqrt{-1}[2\sum_{j=1}^lj\wedge^jW^*-l\wedge_{-1}W^*]\otimes[-2(d-l)+T+T^*-(W+W^*)]z+O(z^2).
\end{align}
We know that
\begin{equation}
\sum_{n\geq 0}a_{n,2}^0z^n=A_0(z)B_0(z),~~\sum_{n\geq 0}a_{n,2}^{\frac{1}{2}}z^n=A_0(z)B_{\frac{1}{2}}(z),~~\sum_{n\geq 0}a_{n,2}^1z^n=A_0(z)B_1(z)+A_1(z)B_0(z),
\end{equation}
then we can get Proposition 4.4 by (4.14)-(4.16).
\end{proof}

\begin{prop}
If $c_1(W)=c_1(M)=0$ and the first Pontrjagin classes $p_1(M)=p_1(W)$ and $p_1(V)=0$ and $M$ is a almost complex $(2d)$-manifold, then\\
1)if either $d-l$ is odd or $d-l< 0$, then
\begin{align}
&a^0_{0,2}=\int_MTd(TM){\rm ch}(\wedge_{-1}(W^*))=0,\\\notag
&a^{\frac{1}{2}}_{0,2}=\int_MTd(TM){\rm ch}(\wedge_{-1}(W^*)\otimes\widetilde{V_C})=0,\\\notag
&
a^1_{0,2}=\int_MTd{\rm ch}(\wedge_{-1}(W^*)\otimes[-2(d-l)-W-W^*+TM\otimes C+\wedge^2\widetilde{V_C}])=0,\\\notag
&a^{\frac{3}{2}}_{0,2}=-\int_MTd(TM){\rm ch}(\wedge_{-1}(W^*)\otimes(\widetilde{V_C}+\wedge^3\widetilde{V_C}))\\\notag
&-\int_MTd(TM){\rm ch}(\wedge_{-1}(W^*)\otimes(-2(d-l)-W-W^*+TM\otimes C)\otimes\widetilde{V_C}])=0.
\end{align}
and associated holomorphic Euler numbers are zero. \\
2)if either $d-l$ is even or $d+1-l< 0$, then $a^0_{1,2}=a^{\frac{1}{2}}_{1,2}=a^1_{1,2}=0.$\\
3)if either $d-l$ is odd or $d+2-l< 0$, then $a^0_{2,2}=a^{\frac{1}{2}}_{2,2}=0.$\\
4)if either $d-l$ is even or $d+3-l< 0$, then $a^0_{3,2}=a^{\frac{1}{2}}_{3,2}=0.$\\
5)if either $d-l$ is odd or $d+4-l< 0$, then $a^0_{4,2}=a^{\frac{1}{2}}_{4,2}=0.$\\
Similarly, we have expressions using holomorphic Euler numbers.
\end{prop}
\begin{thm}If $c_1(W)=c_1(M)=0$ and the first Pontrjagin classes $p_1(M)=p_1(W)$ and $p_1(V)=0$, then\\
1)if $d-l=2$, then
\begin{align}
&-\int_MTd(M)\sum_{p=0}^l(-1)^p{\rm ch}(\wedge^p(W^*)\otimes\widetilde{V_C})=24\int_MTd(M)\sum_{p=0}^l(-1)^p{\rm ch}(\wedge^p(W^*))\\\notag
&\int_MTd{\rm ch}(\wedge_{-1}(W^*)\otimes[-2(d-l)-W-W^*+TM\otimes C+\wedge^2\widetilde{V_C}])\\\notag
&=24\int_MTd(M)\sum_{p=0}^l(-1)^p{\rm ch}(\wedge^p(W^*)).
\end{align}
That is $a^{\frac{1}{2}}_{0,2}=24a^{0}_{0,2}$ and $a^{1}_{0,2}=24a^{0}_{0,2}$ and $a^{\frac{1}{2}}_{0,2}$, $a^{1}_{0,2}$ are the integer multiple of $24$.\\
2)if $d-l=4$, then $240a^{0}_{0,2}+8a^{\frac{1}{2}}_{0,2}-a^{1}_{0,2}=0$.\\
3)if $d-l=6$, then $504a^{0}_{0,2}-32a^{\frac{1}{2}}_{0,2}+a^{1}_{0,2}=0$.\\
4)if $d-l=1$, then $a^{\frac{1}{2}}_{1,2}=24a^{0}_{1,2}$ and $a^{1}_{1,2}=24a^{0}_{1,2}$ .\\
5)if $d-l=3$, then $240a^{0}_{1,2}+8a^{\frac{1}{2}}_{1,2}-a^{1}_{1,2}=0$.\\
6)if $d-l=5$, then $504a^{0}_{1,2}-32a^{\frac{1}{2}}_{1,2}+a^{1}_{1,2}=0$.\\
\end{thm}
\begin{proof}
When $d-l=2$, then $a_{0,2}$ is a modular form of weight $2$ on $\Gamma^0(2)$. By Lemma 2.2, we can get (1). \\
\indent When $d-l=4$, then $a_{0,2}$ is a modular form of weight $4$ on $\Gamma^0(2)$. By Lemma 2.2, we can get $a_{0,2}=\lambda_1\varepsilon_2(\tau)
+\lambda_2(8\delta_2)^2.$ Then compare the coefficients of $1£¬~q^{\frac{1}{2}},~q$ on this equality, then we get (2). Similar, we can get other equalities.
\end{proof}

\subsection{Odd dimensional case}
In this subsection, we use the fundamental setting in Section 3. Let
\begin{equation}
{\rm ch}(Q_1(E),g^{Q_1(E)},d,\tau)=-\frac{2^{\frac{N}{2}}}{8\pi^2}\int^1_0{\rm Tr}\left[g^{-1}dg\left(\frac{\theta'_1(R_u/(4\pi^2),\tau)}{\theta_1(R_u/(4\pi^2),\tau)}\right)\right]du,
\end{equation}
\begin{equation}
{\rm ch}(Q_j(E),g^{Q_1(E)},d,\tau)=-\frac{1}{8\pi^2}\int^1_0{\rm Tr}\left[g^{-1}dg\left(\frac{\theta'_j(R_u/(4\pi^2),\tau)}{\theta_j(R_u/(4\pi^2),\tau)}\right)\right]du,~~j=2,3
\end{equation}
By Proposition 2.2 in \cite{HY}, we have if $c_3(E_C,g,d)=0$, then for any integer $r\geq 1.$ We have ${\rm ch}(Q_j(E),g^{Q_j(E)},d,\tau+1)^{(4r-1)}$ are modular forms of weight $2r$ over $\Gamma_0(2)$, $\Gamma^0(2)$, $\Gamma_{\theta}$ respectively.

\begin{defn}
The two variable elliptic genus of $M^{2d+1}$ with respect to $W$, which we denote by ${\rm Ell}_\alpha(M,W,g,\tau,z)$ for $\alpha=1,2,3$ is defined by
\begin{equation}
{\rm Ell}_\alpha(M,W,g,\tau,z):=\int_M \widehat{A}(M){\rm ch}(E(M,W,\tau,z)){\rm ch}(Q_\alpha(E),g^{Q_\alpha(E)},d,\tau),
\end{equation}
which is the index of the Topelitz operator associated to the Dirac operator twisted by $E(M,W,\tau,z)$ and $g^{Q_\alpha(E)}$.
\end{defn}
Using the same calculations as in Lemma 3.4 in \cite{LiP}, we have
\begin{lem}
\begin{align}
{\rm Ell}_\alpha(M,W,g,\tau,z)=\int_Me^{-\frac{c_1(W)}{2}}\eta(\tau)^{3(d-l)}\prod_{i=1}^d\frac{2\pi \sqrt{-1}x_i}{\theta(\tau,x_i)}\prod_{j=1}^l\theta
(\tau,w_j-z){\rm ch}(Q_\alpha(E),g^{Q_\alpha(E)},d,\tau),
\end{align}
\end{lem}

Similar to Proposition 4.4, we have

\begin{prop}If $c_1(W)=0$ and the first Pontrjagin classes $p_1(M)=p_1(W)$ and the third cohomology group $H^3(M,R)=0$ and $M$ is a simple connected spin manifold, then
the elliptic genus ${\rm Ell}_\alpha(M,W,g,\tau,z)$ is a weak Jacobi form of weight $(d+1)-l$ and index $\frac{l}{2}$ over $\Gamma_0(2)$, $\Gamma^0(2)$, $\Gamma _{\theta}$ respectively.
Then the series $a_{n,1,g}(M,W,V,\tau)$,$a_{n,2,g}(M,W,V,\tau)$, $a_{n,3,g}(M,W,V,\tau)$ determined by
\begin{align}
{\rm exp}(-4\pi^2lG_2(\tau)z^2){\rm Ell}_\alpha(M,W,V,g,\tau,z)=\sum_{n\geq 0}a_{n,\alpha,g}(M,W,V,\tau)\cdot z^n,~~1\leq \alpha\leq 3
\end{align}
are modular forms of weight $d+1-l+n$ over $\Gamma_0(2)$, $\Gamma^0(2)$, $\Gamma _{\theta}$ respectively.
Let \begin{equation} a_{n,2,g}(M,W,V,\tau)=a_{n,2,g}^0+a_{n,2,g}^{\frac{1}{2}}q^{\frac{1}{2}}+a_{n,2,g}^1q+a_{n,2,g}^{\frac{3}{2}}q^{\frac{3}{2}}
+\cdots+a_{n,2,g}^m+\cdots.
\end{equation}
Then $a_{n,2,g}^0=0$, and
\begin{align}
&a_{0,2,g}^{\frac{1}{2}}=-\int_M\widehat{A}(M)\sum_{p=0}^l(-1)^p{\rm ch}(\wedge^p(W^*)){\rm ch}(\widetilde{E_C},g,d),\\\notag
&a_{1,2,g}^{\frac{1}{2}}=-\int_M\widehat{A}(M)\sum_{p=0}^l(-1)^p{\rm ch}(\wedge^p(W^*)){\rm ch}(\widetilde{E_C},g,d)(p-\frac{l}{2})2\pi\sqrt{-1},\\\notag
&a_{2,2,g}^{\frac{1}{2}}=\int_M\widehat{A}(M)\sum_{p=0}^l(-1)^p{\rm ch}(\wedge^p(W^*)){\rm ch}(\widetilde{E_C},g,d)(p-\frac{l}{2})^22\pi^2\\\notag
&-\int_M\widehat{A}(M)\sum_{p=0}^l(-1)^p{\rm ch}(\wedge^p(W^*)){\rm ch}(\widetilde{E_C},g,d)\frac{l}{6}\pi^2
\end{align}
\begin{align}
&a_{1,2,g}^{1}=\int_M\widehat{A}(M)\sum_{p=0}^l(-1)^p{\rm ch}(\wedge^p(W^*)){\rm ch}(\wedge^2\widetilde{V_C},g,d)(p-\frac{l}{2})2\pi\sqrt{-1},\\\notag
&a_{2,2,g}^{1}=\frac{l}{6}\pi^2\int_M\widehat{A}(M)\sum_{p=0}^l(-1)^p{\rm ch}(\wedge^p(W^*)){\rm ch}(\wedge^2\widetilde{V_C},g,d)\\\notag
&+\int_M\widehat{A}(M)\sum_{p=0}^l(-1)^p{\rm ch}(\wedge^p(W^*)){\rm ch}(\wedge^2\widetilde{V_C},g,d)
\frac{1}{2}(p-\frac{l}{2})^2(2\pi\sqrt{-1})^2.
\end{align}
\end{prop}

Similar to Proposition 4.5, we have
\begin{prop}
If $c_1(W)=0$ and the first Pontrjagin classes $p_1(M)=p_1(W)$ and the third cohomology group $H^3(M,R)=0$ and $M$ is a simple connected spin manifold, then\\
1)if either $d-l$ is even or $d-l+1< 0$, then
\begin{align}
&a_{0,2,g}^{\frac{1}{2}}=-\int_M\widehat{A}(M)\sum_{p=0}^l(-1)^p{\rm ch}(\wedge^p(W^*)){\rm ch}(\widetilde{E_C},g,d)=0,
\\\notag
&a_{0,2,g}^{1}=\int_M\widehat{A}(M)\sum_{p=0}^l(-1)^p{\rm ch}(\wedge^p(W^*)){\rm ch}(\wedge^2\widetilde{E_C},g,d)=0,
\end{align}
and associated Toeplitz operators indices are zero. \\
2)if either $d-l$ is odd or $d+2-l< 0$, then $a^{\frac{1}{2}}_{1,2,g}=a^1_{1,2,g}=0.$\\
3)if either $d-l$ is even or $d+3-l< 0$, then $a^{\frac{1}{2}}_{2,2,g}=a^1_{2,2,g}=0.$\\
\end{prop}

Similar to Theorem 4.6, we have

\begin{thm}If $c_1(W)=0$ and the first Pontrjagin classes $p_1(M)=p_1(W)$ and the third cohomology group $H^3(M,R)=0$ and $M$ is a simple connected spin manifold, then\\
1)if $d-l=1$, then $a^{\frac{1}{2}}_{0,2,g}=a^1_{0,2,g}=a^{\frac{3}{2}}_{0,2,g}=0.$
\\
2)if $d-l=3$, then $a^{1}_{0,2,g}=8a^{\frac{1}{2}}_{0,2,g}$.\\
3)if $d-l=5$, then $a^{1}_{0,2,g}=32a^{\frac{1}{2}}_{0,2,g}$.\\
4)if $d-l=0$, then $a^{\frac{1}{2}}_{1,2,g}=a^1_{1,2,g}=0.$
\\
5)if $d-l=2$, then $a^{1}_{1,2,g}=8a^{\frac{1}{2}}_{1,2,g}$.\\
6)if $d-l=4$, then $a^{1}_{1,2,g}=32a^{\frac{1}{2}}_{1,2,g}$.\\
\end{thm}

\section{Two-variable meromorphic elliptic genera which are $\Gamma_0(2)$, $\Gamma^0(2)$, $\Gamma _{\theta}$-Jacobi forms}
\subsection{Even dimensional case}
Let $M^{2d}$ be compact smooth almost complex manifold and $(\overline{\partial}+\overline{\partial}^*)\otimes E$ on $\wedge^{0,*}(T^*M)\rightarrow \wedge^{0,*}(T^*M)$ be the twisted Dirac operator.
Let
\begin{align}
P^z_q(TM)=&\bigotimes _{n=0}^{\infty}
\wedge_{y^{-1}q^{n}}(T^{0,1}(M)-C^d)\otimes\bigotimes _{n=1}^{\infty}
\wedge_{yq^{n}}(T^{1,0}(M)-C^d)\\\notag
&\otimes\bigotimes _{n=1}^{\infty}S_{q^n}(T^{1,0}M-C^d)\otimes\bigotimes _{n=1}^{\infty}S_{q^n}(T^{0,1}M-C^d),\\\notag
Q^z_q(TM)=&\bigotimes _{n=1}^{\infty}
\wedge_{-y^{-1}q^{n-\frac{1}{2}}}(T^{0,1}(M)-C^d)\otimes\bigotimes _{n=1}^{\infty}
\wedge_{-yq^{n-\frac{1}{2}}}(T^{1,0}(M)-C^d)\\\notag
&\otimes\bigotimes _{n=1}^{\infty}S_{q^n}(T^{1,0}M-C^d)\otimes\bigotimes _{n=1}^{\infty}S_{q^n}(T^{0,1}M-C^d),\\\notag
R^z_q(TM)=&\bigotimes _{n=1}^{\infty}
\wedge_{y^{-1}q^{n-\frac{1}{2}}}(T^{0,1}(M)-C^d)\otimes\bigotimes _{n=1}^{\infty}
\wedge_{yq^{n-\frac{1}{2}}}(T^{1,0}(M)-C^d)\\\notag
&\otimes\bigotimes _{n=1}^{\infty}S_{q^n}(T^{1,0}M-C^d)\otimes\bigotimes _{n=1}^{\infty}S_{q^n}(T^{0,1}M-C^d).
\end{align}
Let $2\pi\sqrt{-1}x_j$, $1\leq j \leq d$ be the Chern root of $T^{1,0}M$, then
\begin{equation}
{\rm Ind}((\overline{\partial}+\overline{\partial}^*)\otimes P^z_q(TM))=\int_M\prod_{j=1}^dx_j\frac{\theta'(0,\tau)}{\theta_1(z,\tau)}
\frac{\theta_1(x_j+z,\tau)}{\theta(x_j,\tau)}:=\Phi_L.
\end{equation}
If we assume that $M$ is spin, then
\begin{align}
&{\rm Ind}(D\otimes Q^z_q(TM))=\int_M\prod_{j=1}^dx_j\frac{\theta'(0,\tau)}{\theta_2(z,\tau)}
\frac{\theta_2(x_j+z,\tau)}{\theta(x_j,\tau)}:=\Phi_W,\\\notag
&{\rm Ind}(D\otimes R^z_q(TM))=\int_M\prod_{j=1}^dx_j\frac{\theta'(0,\tau)}{\theta_3(z,\tau)}
\frac{\theta_3(x_j+z,\tau)}{\theta(x_j,\tau)}:=\Phi_W^*.
\end{align}

\begin{prop}If $c_1(M)=0$, then
the elliptic genus $\Phi_L,~ \Phi_W, ~\Phi_W^*$ are weak meromorphic Jacobi forms of weight $d$ and index $0$ over $\Gamma_0(2)$, $\Gamma^0(2)$, $\Gamma _{\theta}$ respectively.
\end{prop}
\begin{proof} This comes from the following equalities:
\begin{align}
&\Phi_L(z,\tau+1)=\Phi_L(z,\tau),~~\Phi_L(\frac{z}{\tau},-\frac{1}{\tau})=\tau^d\Phi_W(z,\tau),\\\notag
&\Phi_W(z,\tau+1)=\Phi^*_W(z,\tau),~~\Phi_W(\frac{z}{\tau},-\frac{1}{\tau})=\tau^d\Phi_L(z,\tau),\\\notag
&\Phi_W^*(z,\tau+1)=\Phi_W(z,\tau),~~\Phi_W^*(\frac{z}{\tau},-\frac{1}{\tau})=\tau^d\Phi_W^*(z,\tau).
\end{align}
\end{proof}
\subsection{Odd dimensional case}
Let $M$ be a spin $(2d+1)$-manifold and $TM\otimes C=T^{1,0}M\oplus T^{0,1}M\oplus C$ and $c_1(T^{1,0}M)=0$. Consider the indices of the following Toeplitz operators
\begin{equation}
{\rm Ind}(T_{D\otimes P^z_q(TM),g^{Q_1(E)}})=2^{\frac{N}{2}}\int_M\prod_{j=1}^dx_j\frac{\theta'(0,\tau)}{\theta_1(z,\tau)}
\frac{\theta_1(x_j+z,\tau)}{\theta(x_j,\tau)}{\rm ch}(Q_1(E),g^{Q_1(E)}£¬d,\tau):=\Phi_{L,g}.
\end{equation}

\begin{align}
&2^{\frac{N}{2}}{\rm Ind}(T_{D\otimes Q^z_q(TM),g^{Q_2(E)}})=2^{\frac{N}{2}}\int_M\prod_{j=1}^dx_j\frac{\theta'(0,\tau)}{\theta_2(z,\tau)}
\frac{\theta_2(x_j+z,\tau)}{\theta(x_j,\tau)}{\rm ch}(Q_2(E),g^{Q_2(E)}£¬d,\tau):=\Phi_{W,g},\\\notag
&2^{\frac{N}{2}}{\rm Ind}(T_{D\otimes R^z_q(TM),g^{Q_3(E)}})=2^{\frac{N}{2}}\int_M\prod_{j=1}^dx_j\frac{\theta'(0,\tau)}{\theta_3(z,\tau)}
\frac{\theta_3(x_j+z,\tau)}{\theta(x_j,\tau)}{\rm ch}(Q_3(E),g^{Q_3(E)}£¬d,\tau):=\Phi_{W,g}^*.
\end{align}

\begin{prop}If $c_1(T^{1,0}M)=0$, then
the elliptic genus $\Phi_{L,g},~ \Phi_{W,g}, ~\Phi_{W,g}^*$ are weak meromorphic Jacobi forms of weight $d+1$ and index $0$ over $\Gamma_0(2)$, $\Gamma^0(2)$, $\Gamma _{\theta}$ respectively.
\end{prop}

\section{Acknowledgements}

 The author was supported by Science and Technology Development Plan Project of Jilin Province, China: No.20260102245JC and NSFC No.11771070.
\vskip 1 true cm
Data availability: This manuscript has no associated data.\\
 \indent  Declarations Conflict of interest: This manuscript has no conflict of interest.


\bigskip
\bigskip

 \indent{School of Mathematics and Statistics,
Northeast Normal University, Changchun Jilin, 130024, China }\\
\indent E-mail: {\it wangy581@nenu.edu.cn }\\

\end{document}